\documentclass{amsart}

\newtheorem{theorem}{Theorem}[section]
\newtheorem{lemma}[theorem]{Lemma}
\newtheorem{cor}[theorem]{Corollary}

\theoremstyle{definition}
\newtheorem{definition}[theorem]{Definition}
\newtheorem{example}[theorem]{Example}

\theoremstyle{remark}

\usepackage{latexsym,amsfonts}
\numberwithin{equation}{section}

\def\co{\colon\thinspace}

\newcommand{\nser}{\mathcal{N}} 
\newcommand{\lnser}{\mathcal{N}^l} 
\newcommand{\rnser}{\mathcal{N}^r} 
\newcommand{\cser}{\mathcal{C}}
\newcommand{\lcent}{\mathcal{Z}^l} 

\DeclareMathOperator{\im}{im}

\DeclareMathOperator{\soc}{Soc}

\DeclareMathOperator{\End}{End}

\DeclareMathOperator{\alg}{alg\langle}

\newcommand{\snr}{\mbox{$\triangleleft\mspace{-1.8mu}\triangleleft_r\medspace$}}

\DeclareMathOperator{\ideq}{\unlhd}

\begin{document}

\title{Engel subalgebras of Leibniz algebras} 
\author{Donald W. Barnes}
\address{1 Little Wonga Rd, Cremorne NSW 2090 Australia}
\email{donwb@iprimus.com.au}
\thanks{This work was done while the author was an Honorary Associate of the
School of Mathematics and Statistics, University of Sydney.}

\begin{abstract} Engel subalgebras of finite-dimensional Leibniz algebras are shown to have similar properties to those of Lie algebras.  Using these, it is shown that a left Leibniz algebra, all of whose maximal subalgebras are right ideals, is nilpotent.  A primitive Leibniz algebra is shown to split over its minimal ideal and that all the complements to its minimal ideal are conjugate.   A subalgebra is shown to be a Cartan subalgebra if  and only if it is minimal Engel, provided that the field has sufficiently many elements.  Cartan subalgebras are shown to have a property analogous to intravariance.  
\end{abstract}
\maketitle

\section{Introduction}\label{sec-intro}
Engel subalgebras, so named because of their close connection with Engel's Theorem, have been found useful in the study of Lie algebras.  In this paper, I show that Engel subalgebras of Leibniz algebras have similar properties and use them to prove Leibniz algebra generalisations of some theorems on Lie algebras.  All the algebras considered in this paper are finite-dimensional over a field $F$ which, unless otherwise stated, may be of any characteristic.

In this section, I set out the basic definitions and some basic results on Leibniz algebras.  In Section \ref{sec-nilp}, I set out the basic properties of soluble and nilpotent Leibniz algebras and of their representations.   In Section \ref{sec-eng}, I establish the properties of Engel subalgebras and use them to prove some analogues of known theorems on Lie algebras.  In Section \ref{sec-cartan}, I show that, provided that the field has at least $\dim(A) + 1$ elements, the Cartan subalgebras of the Leibniz algebra $A$ are precisely its minimal Engel subalgebras.  I show that Cartan subalgebras have a property which could be regarded as the Leibniz algebra analogue of the group theory concept of intravariance.  (A subgroup $U$ of a group $G$ is called {\em intravariant} in $G$ if every automorphism of $G$ maps $U$ onto a conjugate subgroup.)

Let $A$ be an algebra over the field $F$.  We denote by $L_a$ the  left multiplication by $a \in A$, thus $L_a(x) = ax$ for all $x \in A$.  Likewise, we denote by $R_a$ the right multiplication by $a$, thus $R_a(x) = xa$.

\begin{definition} A (left) Leibniz algebra is an algebra $A$ for which all the left multiplications are derivations, that is, $$a(bc) = (ab)c + b(ac)$$
for all $a,b,c \in A$. \end{definition}

Observe that $L_{ab} = L_a L_b - L_b L_a$ for all $a,b \in A$.  Thus the $L_a$ form a Lie algebra of linear transformations of $A$.

\begin{theorem}\label{power} Let $A$ be a Leibniz algebra and let $a \in A$.  Let $b$ be a product of $n > 1$ $a$'s.  Then $L_b = 0$, however the product $b$ is associated. \end{theorem}

\begin{proof}  We have $a(ax) = (aa)x + a(ax)$, so $(aa)x = 0$ for all $x \in A$.  Thus $L_{(aa)} = 0$.

Now suppose $n > 2$.  Then $b = uv$ and $bx = (uv)x = u(vx) - v(ux)$, so $L_b = L_uL_v - L_vL_u = 0$ since at least one of $u,v$ has degree $> 1$.  
\end{proof}

\begin{cor} Let $b \ne 0$ be a power of $a$.  Then $b$ has the form $b=a(a( \dots (aa) \dots ))$. \end{cor}

\begin{proof}  We have $b = uv$ where $u,v$ are non-zero powers of $a$.  By induction  over the degree, $v$ has  the asserted form.  Since $L_u \ne 0$, $u$ has degree $1$, thus   $b = av$ has the asserted  form. \end{proof}

As a result of this, powers of elements are unambiguous.  We can define $a^n$ by setting $a^1 = a$ and $a^{n+1} = aa^n$.    We can similarly define powers $A^n$ of the algebra $A$ by $A^1 = A$ and $A^{n+1} = A A^n$.  By a theorem of Ayupov and Omirov \cite{AyO}, if $b$ is a product of $n$ elements of $A$, however bracketed, then $b \in A^n$.  They also point out that the left centre $\lcent(A) = \{z \in A \mid za=0 \text{ for all } a \in A\}$ is a $2$-sided ideal of $A$, that $a^2 \in \lcent(A)$ for all $a \in A$ and that $A/ \lcent(A)$ is a Lie algebra.   It follows from this that a simple Leibniz algebra is necessarily a Lie algebra.

\begin{definition}  We say that the Leibniz algebra $A$ is abelian if $A^2 = 0$. \end{definition}
Note that if $A = \alg a \rangle$ is the algebra generated by the single element $a$, then $\langle a^n \mid n>1 \rangle$ is an abelian subalgebra and, if $A \ne 0$, then $A^2 \ne A$.  In particular, a $1$-dimensional algebra is abelian.  We write $U \le A$ for $U$ is a subalgebra of $A$ and $U \ideq A$ for $U$ is an  ideal of $A$.

\begin{definition} Let $U$ be a subalgebra of the Leibniz algebra $A$.  The left normaliser of $U$ in $A$ is the subset $\lnser_A(U) = \{a \in A \mid au \in U \text{ for all } u \in U\}$, the right normaliser the subset $\rnser_A(U) = \{a \in A \mid ua \in U \text{ for all } u \in U\}$ and the normaliser is the subset $\nser_A(U) = \{a \in A \mid au \in U \text{ and } ua \in U \text{ for all } u \in U\}$.
\end{definition}

It easily seen that the normaliser $\nser_A(U)$ and the left normaliser $\lnser_A(U)$ of a subalgebra $U$ are subalgebras, but the right normaliser $\rnser_A(U)$ need not be.

\begin{example} Let $A = \langle u, n, k, n^2 \rangle$ with the multiplication given by $un=u, nu=-u+k, un^2 =k, u^2=0, uk=0, nk=-k, n^3=0$ with $ka = n^2a = 0$ for all $a \in A$.  Put $U=\langle u \rangle$.  Then $A$ is a Leibniz algebra, $U$ is a subalgebra and $\rnser_A(U)$ is not a subalgebra of $A$ as $n \in \rnser_A(U)$ but $n^2 \notin \rnser_A(U)$. \end{example}

\begin{definition} A bimodule of a Leibniz algebra $A$ is a vector space $M$ with two bilinear
compositions $am, ma$ for $a \in A$ and $m \in M$ such that \begin{equation*}\begin{split}
a(bm) &= (ab)m + b(am)\\
a(mb) &= (am)b + m(ab)\\
m(ab) &= (ma)b + a(mb)\\ \end{split} \end{equation*}
for all $a,b \in A$ and $m \in M$. \end{definition}
These are precisely the conditions which would be satisfied if $A$ and $M$ were contained in some Leibniz algebra.  A bimodule $M$ is equivalent to a pair $(S,T)$ of linear map into the endomorphism algebra $\End(M)$ of $M$, $T_a(M) = am$ and $S_a(m) = ma$ for all $a \in A$ and $m \in M$.  The maps $(S, T)$ satisfy
\begin{equation*}\begin{split}
T_a \circ T_b &= T_{ab} + T_b \circ T_a\\
T_a \circ S_b &= S_b \circ T_a + S_{ab}\\
S_{ab} &= S_b \circ S_a + T_a \circ S_b\\
\end{split}\end{equation*}
The first of these shows that $M$ is a module for the Lie algebra $T_A = \{T_a \mid a \in A \}$.  Note that, if $A$ is a Lie algebra and $T \co A \to \End(M)$ is a Lie representation, setting either $S = -T$ or $S=0$ makes $M$ into a bimodule.

\section{Solubility and nilpotency} \label{sec-nilp}
\begin{definition} The Leibniz algebra $A$ is called nilpotent if, for some $k$, we have $A^k = 0$. \end{definition}
The following result is proved exactly as is the analogous result for groups and for Lie algebras.
\begin{lemma}\label{lem-normNilp} Suppose $A$ is nilpotent.  Let $U \ne A$ be a subalgebra of $A$.  Then $\nser_A(U) \ne U$. \end{lemma}

We shall make use of the analogues of Engel's Theorem for Lie algebras of linear transformations and for abstract Lie algebras.   Ayupov and Omirov \cite[Theorem 2]{AyO} prove: 
\begin{theorem}\label{th-absEng} Let $A$ be a finite-dimensional  Leibniz algebra.  Suppose that
$L_a$ is nilpotent for all $a \in A$.  Then $A$ is nilpotent. \end{theorem}

Patsourakos obtains this as a corollary of the stronger result \cite[Theorem 7]{Pats}

\begin{theorem}\label{th-linEng}  Let $A$ be a finite-dimensional Leibniz algebra and let $(S,T)$ be
a representation of $A$ on a vector space $M\ne 0$ such that $T_a$ is nilpotent for all $a \in A$.  Then $S_a$ is nilpotent for all $a \in A$ and there exists $m \in M$, $m \ne 0$ such that $T_a(m) =
S_a(m) = 0$ for all $a \in A$.\end{theorem} 

We define the derived series $A^{(r)}$ by $A^{(0)} = A$ and $A^{(r+1)} = (A^{(r)})^2$.

\begin{definition}  We say that the Leibniz algebra $A$ is soluble if there exists a chain of subalgebras $A = A_0 \supseteq A_1 \supseteq \dots \supseteq A_n=0$ such that $A_{i+1}$ is a $2$-sided ideal in $A_i$ and $A_i/A_{i+1}$ is abelian. \end{definition} 
Thus $A$ is soluble if and only if $A^{(r)} = 0$ for some $r$.  Ayupov and Omirov have proved \cite[Theorem 4]{AyO}:
\begin{theorem}\label{th-ch0}  Suppose $A$ is a soluble Leibniz algebra over a field $F$ of
characteristic $0$.  Then $A^2$ is nilpotent. \end{theorem}

\section{Engel subalgebras}\label{sec-eng}
Let $a \in A$ and set $E_A(a) = \{x \in A \mid L_a^n(x) = 0 \text{ for some $n$}\}$.  As is the case for Lie algebras,  $E_A(a)$ is a subalgebra of $A$.   We call such subalgebras Engel subalgebras.  Unlike the Lie algebra case, we need not have $a \in E_A(a)$.  However, we do have

\begin{lemma}\label{lem-inE}  For any $a \in A$, there exists $a' \in E_A(a)$ such that $E_A(a') = E_A(a)$. \end{lemma}

\begin{proof} Put $U = \alg a \rangle$ and let $\mu = L_a|U \to U$.  For sufficiently large $r$,
we have $U = \im(\mu^r) \oplus \ker(\mu^r)$.  Thus there exist $a' \in \ker(\mu^r)$ and $b \in U^2$ such that $a = a' + b$.  Then $a' \in E_A(a)$, and $L_{a'} = L_a$ since $L_b = 0$. \end{proof}

\begin{lemma} \label{selfN}Suppose $U$ is a subalgebra of $A$ and $E_A(a) \subseteq U$. 
Then $\rnser_A(U) = U$. \end{lemma}

\begin{proof} By Lemma \ref{lem-inE}, we may suppose $a \in E_A(a)$.  Suppose $b \in
\rnser_A(U)$.  Then $ab \in U$.  Since $L_a$ acts invertibly on $A/E_A(a)$, this implies $b \in
U$. \end{proof}

\begin{theorem} Let $A$ be a finite-dimensional Leibniz algebra and suppose every maximal
subalgebra of $A$ is a right ideal.  Then $A$ is nilpotent.\end{theorem}

\begin{proof}   Suppose $E_A(a) \ne A$.  Then there exists a maximal subalgebra $M \supseteq
E_A(a)$.    By Lemma \ref{selfN}, we have $\rnser_A(M) = M$ contrary to hypothesis.  Therefore  $E_A(a) = A$ for all $a \in A$.  By Theorem \ref{th-absEng}, $A$ is nilpotent.
\end{proof}

\begin{definition}  We say that the subalgebra $U$ is right subnormal in $A$, written $U \snr A$ if there exists a chain of subalgebras $U_0 = A \supseteq U_1 \supseteq \dots \supseteq U_k = U$ with each $U_i$ a right ideal in $U_{i-1}$. \end{definition}

\begin{definition}  The Frattini subalgebra $\Phi(A)$ of the algebra $A$ is the intersection of the
maximal subalgebras of $A$.\end{definition}

\begin{theorem}\label{th-FratQ} Suppose that $U \snr A$, $V$ is an ideal of $U$ and that $V
\subseteq \Phi(A)$.  Suppose that $U/V$ is nilpotent.  Then $U$ is nilpotent. \end{theorem}

\begin{proof} Let $u \in U$ and let $U_0 = A \supseteq U_1 \supseteq \dots \supseteq U_k = U$ be a chain of subalgebras with each $U_i$ a right ideal in $U_{i-1}$.  Then $L_u^k(A) \subseteq U$.  Since $U/V$ is nilpotent, for some $r$, we have $L_u^r(U) \subseteq V$.  Thus $L_u^{k+r}(A) \subseteq \Phi(A)$.  But $L_u^{k+r}(A) + E_A(u) = A$.  Therefore $E_A(u) = A$.  By Theorem \ref{th-absEng}, $U$ is nilpotent.
\end{proof}

\begin{cor} Suppose $U \subseteq \Phi(A)$ is a right ideal of $A$.  Then $U$ is nilpotent. \end{cor}

\begin{definition} The centraliser of the subalgebra $U$ in $A$ is the subspace $$\cser_A(U) = \{x \in A \mid xU = Ux = 0\}.$$ \end{definition}

Clearly, $\cser_A(U) \le A$ and, if $U \ideq A$ then $\cser_A(U) \ideq A$.
\begin{definition} A soluble algebra $A$ is called primitive if it has a minimal ideal $C$ such that $\cser_A(C) = C$.\end{definition}
\begin{definition}  The socle $\soc(A)$ of the algebra $A$ is the sum of the minimal ideals of $A$. \end{definition}
Clearly, a primitive algebra has only one minimal ideal and the socle is that unique minimal deal.

\begin{lemma}\label{lem-primsoc} Let $P$ be a primitive Leibniz algebra which is not a Lie algebra.  Then $\soc(P) = \lcent(P)$.  \end{lemma}

\begin{proof} $\lcent(P)$ is a non-zero ideal, so contains a minimal ideal.  As $\soc(P)$ is the only minimal ideal, $\soc(P) \subseteq \lcent(P)$.  Since $\cser_A(\soc(P)) = \soc(P)$ and $\lcent(P)$ is abelian, $\soc(P) = \lcent(P)$. \end{proof}

\begin{lemma}\label{lem-prim} Let $P$ be a primitive algebra and let $C = \soc(P)$.  Then there exists a maximal subalgebra $M$ such that $M+C=P$ and $M \cap C = 0$. \end{lemma}

\begin{proof}  We may suppose that $P \ne C$.  Let $B/C$ be a minimal ideal of $P/C$.  Since $P$ is soluble, $B^2 \subseteq C$.  The $B^r$ are ideals of $P$.  As $\cser_P(C)= C$, $B^2 \ne 0$, so $B^2 = C$.  If $BC=0$, then $CB = B^2B \subseteq BB^2 =0$, so $BC \ne 0$.  Thus there exists $b \in B$ such that $L_b(C) \ne 0$.  Put $M = E_P(b)$.  Since $L_b^2(P) \subseteq C$, we have $M + C = P$.  But $M \cap C$ is an ideal of $P$ and $M \not\supseteq C$ so $M \cap C = 0$.
\end{proof}

If $C$ is an abelian ideal of the Leibniz algebra $A$ and $c \in C$, then the map $\alpha_c = (1 + L_c) \co A \to A$ is an automorphism.  The subalgebras $U,V$ are said to be $C$-conjugate if $V = \alpha_c(U)$ for some $c \in C$.  For a primitive Lie algebra $P$ with socle $C$, all complements to $C$ in $P$ are $C$-conjugate by \cite[Theorem 1.1]{BN}.  Somewhat surprisingly, this is also true for primitive Leibniz algebras $P$ with $\soc(P) = \lcent(P)$.

\begin{lemma}\label{lem-single}  Let $P$ be a primitive Leibniz algebra with socle $C = \lcent(P)$.   Then there is only one complement to $C$ in $P$. \end{lemma}

\begin{proof}  Let $U,V$ be complements to $C$.  The map $\alpha \co U \to V$ which maps $u \in U$ to the unique element of $V$ in the coset $u+C$ is an isomorphism.  Let $f \co U \to C$ be given by $f(u) = \alpha(u) - u$.  Since $\alpha$ is an isomorphism and $f(u) \in \lcent(P)$, for $u,u' \in U$, we have
$$uu' + f(uu') = (u+f(u))(u' + f(u')) = uu' + u f(u').$$
Thus $f \co U \to C$ is a $U$-module homomorphism.  Suppose $f \ne 0$ and let $U_0 = \ker(f)$.  Since $C$ is an irreducible $U$-module, $U_0$ is a maximal ideal of $U$.  Thus $U_0 \supseteq U^2$ and $U$ acts trivially on $U/U_0$.  But the action of $U$ on  $C$ is faithful.  It follows that $f = 0$ and $U=V$.\end{proof}

We thus have
\begin{theorem}\label{th-conjP}  Let $P$ be a primitive Leibniz algebra with socle $C$.  Then $P$ splits over $C$ and all complements to $C$ in $P$ are $C$-conjugate. \end{theorem}

\section{Cartan subalgebras}\label{sec-cartan}
\begin{definition} A Cartan subalgebra of the Leibniz algebra $A$ is a nilpotent subalgebra $C$ such that $C = \nser_A(C)$. \end{definition}

The definition used in the literature (for example, see Omirov \cite{Om}) imposes the stronger requirement that $\rnser_A(C)= C$.  The following lemma shows that this is unnecessary.

\begin{lemma}\label{lem-cart} Let $C$ be a Cartan subalgebra of $A$ and let $U$ be a subalgebra of $A$ which contains $C$.  Then $\rnser_A(U) = U$. \end{lemma}

\begin{proof}  Let $E = \bigcap_{c \in C} E_A(c)$.  Then $E \ge C$ is a subalgebra of $A$, so a submodule of the $C$-bimodule $A$.  All the $L_c$ for $c \in C$ act nilpotently on $E/C$.   If $E \ne C$, then by Theorem \ref{th-linEng}, there exists $e \in E$, $e \not\in C$ such that $L_c(e), R_c(e) \in C$ for all $c \in C$, that is, $e \in \nser_A(C)$ contrary to hypothesis.  Thus $E = C$. 

We consider $A$ as a module for the Lie algebra $L = \{L_c \mid c \in C\}$.  Since $C$ is nilpotent, $L$ is nilpotent and $A$ has a submodule $B$ such that $L$ acts trivially on every composition factor of $B$ and non-trivially on every composition factor of $A/B$.  From $E = C$, it follows that $B = C$.  Thus $A/U$ has no composition factor on which $L$ acts trivially, so $\rnser_A(U) = U$.
\end{proof}

\begin{cor}\label{cor-cart}  Let $C$ be a Cartan subalgebra of $A$ and let $K \ideq A$.  Then $C+ K/K$ is a Cartan subalgebra of $A/K$. \end{cor}
\begin{proof} $C+K/K$ is nilpotent and, by Lemma \ref{lem-cart}, $\nser_A(C+K) = C+K$.
\end{proof} 

\begin{lemma}\label{lem-minE} Let $A$ be a Leibniz algebra of dimension $n$ over the field $F$ of at least $n+1$ elements.  Let $U$ be a subalgebra of $A$ and suppose that $E = E_A(u_0)$ is minimal in the set $\{E_A(u) \mid u \in U\}$.  Suppose $E \supseteq U$.  Then $E_A(u) \supseteq U$ for all $u \in U$. \end{lemma}

\begin{proof} We consider $A$ as a module for the algebra generated by the $L_u$.  Since $E$ is a subalgebra of $A$ and $U \subseteq E$, $E$ is a submodule.  Take any $u_1 \in U$ and let $\theta(x,t), \phi(x,t)$ and $\psi(x,t)$ be the characteristic polynomials in the indeterminate $x$ of $L_{u_o + t u_1}$ on $A, E$ and $A/E$ respectively.  Then
\begin{equation*}\begin{split}
\theta(x,t) &= \phi(x,t) \psi(x,t),\\
\phi(x,t) &= x^r + \alpha_1(t)x^{r-1} + \dots + \alpha_r(t),\\
\psi(x,t) &= x^{n-r} + \beta_1(t)x^{n-r-1} + \dots + \beta_{n-r}(t),\\  
\end{split}\end{equation*}
where $r= \dim(E)$ and $\alpha_i(t), \beta_i(t)$ are polynomials in $t$ of degree at most $i$.  We prove that $\alpha_i = 0$ for all $i$.

Since $0$ is not an eigenvalue of $L_{u_0}$ on $A/E$, $\beta_{n-r}(0) \ne 0$.  Thus $\beta_{n-r}(t)$ is not the zero polynomial.  Since $\beta_{n-r}(t)$ has at most $n-r$ roots in $F$, there exist $r+1$ distinct elements $t_1, \dots, t_{r+1} \in F$ such that $\beta_{n-r}(t_j) \ne 0$.  But $\beta_{n-r}(t_j) \ne 0$ implies that $E_A(u_0 + t_j u_1) \subseteq E$.  Therefore $E_A(u_0 + t_j u_1) = E$ by the minimality of $E$.  But this implies $\phi(x, t_j) = x^r$, that is, $\alpha_i(t_j) = 0$ for all $j$.  Since $\alpha_i(t)$ has degree at most $i < r+1$, $\alpha_i(t)$ is the zero polynomial.  Thus $E_A(u_0 + tu_1) \supseteq E$ for all $u_1 \in U$ and all $t \in F$.  Given $u \in U$, put $u_1 = u_0 - u$.  Then $E_A(u) = E_A(u_0+u_1) \supseteq E$.
\end{proof}

\begin{theorem}\label{th-minE}  Let $A$ be a Leibniz algebra of dimension $n$ over the field $F$ of at least $n+1$ elements.  Let $U$ be a subalgebra of $A$.  Then $U$ is a Cartan subalgebra of $A$ if and only if $U$ is minimal in the set of Engel subalgebras of $A$. \end{theorem} 

\begin{proof}  Suppose $U = E_A(u_0)$ is minimal Engel in $A$.  By Lemma \ref{lem-inE}, we may supose that $u_0 \in U$.  Then $\rnser_A(U) = U$ by Lemma \ref{selfN}.  For any $u \in U$, $E_A(u) \supseteq U$ by Lemma \ref{lem-minE}. Thus $L_u| U \to U$ is nilpotent for all $u \in U$, $U$ is nilpotent by Theorem \ref{th-absEng}, and $U$ is a Cartan subalgebra.

Suppose conversely, that $U$ is a Cartan subalgebra of $A$.  For any $u \in U$, we have $E_A(u) \supseteq U$.  We have to show that there exists $u_0 \in U$ such that $E_A(u_0) = U$.  Take $u_0 \in U$ such that $E = E_A(u_0)$ is minimal in the set $\{E_A(u) \mid u \in U\}$.   By Lemma \ref{lem-minE}, $E_A(u) \supseteq E$ for all $u \in U$.  Thus all $L_u$ act nilpotently on $E/U$.  If $E \ne U$, by Theorem \ref{th-linEng}  there exists  $\bar{e} \in E/U$, $\bar{e} \ne 0$, such that $L_u(\bar{e}) = 0$ for all $u \in U$, that is, $e \not\in U$ such that $e \in \rnser_A(U)$. 
\end{proof}

A subalgebra $U$ of a Lie algebra $L$ is said to be  intravariant in $L$ if every derivation of $L$ is the sum of an inner derivation and a derivation which stabilises $U$.  This definition seems inappropriate for Leibniz algebras as it deals only with left actions on the algebra, ignoring right actions. However, Cartan subalgebras of Leibniz algebras have a property which for Lie algebras, by Barnes \cite[Lemma 1.2]{Frat}, is equivalent to intravariance.

\begin{theorem} Let $N$ be an ideal of the Leibniz algebra $A$ and let $C$ be a Cartan subalgebra of $N$.  Then $N + \nser_A(C) = A$.\end{theorem}

\begin{proof}  Let $E = \bigcap\{E_A(c) \mid c \in C\}$.  Then $E\supseteq C$ is a subalgebra of $A$.   By Lemma \ref{selfN}, $\rnser_N(C) = C$.   But  the Lie algebra $L_C$ acts nilpotently on $E \cap N$.  Therefore $E \cap N = C$.  However, for all $c \in C$ and $e \in E$, we have $ce, ec \in E \cap N$.  Thus $E = \nser_A(C)$.  As $L_C$-module, $A$ is (uniquely) the direct sum of a submodule on which $L_C$ acts nilpotently (in this case, $E$) and a submodule $K$ with no trivial composition factors.  Since $L_CA \subseteq N$, we have $K \subseteq N$ and $E + N = A$. \end{proof}

\bibliographystyle{amsplain}

\end{document}